\documentclass[12pt]{article}
\usepackage{amsmath, amssymb, amsfonts}
\usepackage{amsthm}
\usepackage{geometry}
\usepackage{tikz}
\usetikzlibrary{positioning,calc,arrows.meta}
\usepackage{cite}
\usepackage[colorlinks=true,linkcolor=blue,citecolor=blue,urlcolor=blue]{hyperref}

\numberwithin{equation}{section}

\newtheorem{theorem}{Theorem}[section]

\newtheorem{proposition}{Proposition}[section]
\newtheorem{corollary}{Corollary}[section]

\theoremstyle{definition}
\newtheorem{definition}{Definition}[section]
\newtheorem{remark}{Remark}[section]

\title{\textbf{Universality of shallow and deep neural networks on non-Euclidean spaces}}

\author{\textsc{Vugar Ismailov}\thanks{The author can be contacted at
\texttt{vugaris@mail.ru} or \texttt{vugaris@gmail.com}.}}

\date{}

\begin{document}
\maketitle

\begin{abstract}
We study shallow and deep neural networks whose inputs range over a general
topological space. The model is built from a prescribed family of continuous
feature maps and reduces to multilayer feedforward networks in the Euclidean
case. We focus on the universal approximation property and establish general
conditions under which such networks are dense in spaces of continuous
vector-valued functions on arbitrary topological spaces and, 
in particular, locally convex spaces. Universality results obtained in the 
arbitrary-width case extend classical approximation theorems to 
non-Euclidean spaces. We also consider the
deep narrow setting, in which the width of each hidden layer is uniformly
bounded while the depth is allowed to grow. We identify conditions under which
such networks retain the universal approximation property. As a concrete example, we
employ Ostrand’s extension of the Kolmogorov superposition theorem to derive an
explicit universality result for products of compact metric spaces, with width
bounds expressed in terms of topological dimension.

\end{abstract}


\section{Introduction}

Neural networks play a central role in modern machine learning and artificial 
intelligence. Among the many architectures in use, \emph{multilayer
feedforward neural networks} remain a basic and influential model 
in both theory and practice. A broad overview of 
the approximation-theoretic properties of feedforward
neural networks can be found, for example, in \cite{Pinkus,Ism}, while modern
deep learning practice is surveyed in \cite{Good}.

A multilayer feedforward neural network consists of an input layer, an output
layer, and one or more hidden layers arranged sequentially. Information
propagates forward through the network: each neuron forms an affine
combination of the outputs of the previous layer and then applies a fixed
nonlinear \emph{activation function}. This repeated affine--nonlinear
processing enables the network to represent complex nonlinear
mappings.

The approximation theory of neural networks is rooted in the
\emph{universal approximation property}, also called the density
property. In the classical formulation, for suitable activation functions
$\sigma$, the linear span of functions of the form
$\sigma(\mathbf{w}\cdot \mathbf{x}-\theta)$ is dense in $C(K)$ for every
compact set $K\subset\mathbb{R}^d$. A fundamental result in this direction
states that, when the activation function $\sigma$ is assumed to be
continuous, the universal approximation property holds if and only if
$\sigma$ is not a polynomial (see \cite{Pinkus,Leshno}).

Note that the inner product $\mathbf{w}\cdot\mathbf{x}$ appearing in classical
neural networks is a linear continuous functional on $\mathbb{R}^d$.
Conversely, by the Riesz representation theorem, every linear continuous
functional on $\mathbb{R}^d$ can be written in this form
(see \cite[Theorem~13.32]{Roman}). Let $\mathcal{L}(\mathbb{R}^d)$ denote the
space of all linear continuous functionals on $\mathbb{R}^d$.
Under this formulation, a single hidden layer network applies a collection of
functions from $\mathcal{L}(\mathbb{R}^d)$ to the input, adds scalar biases,
and then applies a nonlinearity. In particular, 
the classical universal approximation property can be
rephrased as the density of the function class
\[
\mathcal{M}(\sigma )=\operatorname{span}
\bigl\{
\sigma(f(x)-\theta)
:
f\in\mathcal{L}(\mathbb{R}^d),\ \theta\in\mathbb{R}
\bigr\}
\]
in $C(K)$ for every compact set $K\subset\mathbb{R}^d$.

\medskip
\noindent
\textbf{From Euclidean inputs to general topological inputs.}
The preceding reformulation highlights that, in standard feedforward neural
networks, the only structural use of the Euclidean input space
$\mathbb{R}^d$ lies in the availability of a rich family of scalar-valued
continuous functions, namely, linear functionals, applied to the input.
This observation suggests a natural extension beyond Euclidean spaces.

Let $X$ be an arbitrary topological space. Instead of linear functionals,
we fix a family $\mathcal{A}(X)\subset C(X)$ of continuous real-valued
functions on $X$, which play the role of admissible scalar feature maps
available to the network. In the classical case $X=\mathbb{R}^d$, the choice
$\mathcal{A}(X)=\mathcal{L}(\mathbb{R}^d)$ recovers exactly the standard
architecture, since each feature map $f\in\mathcal{A}(X)$ corresponds to an
inner product $\mathbf{w}\cdot \mathbf{x}$.

\begin{definition}[Basic family]
A set $\mathcal{A}(X)\subset C(X)$ is called a \emph{basic family} (for
feedforward neural networks on $X$) if its elements serve as the admissible
scalar feature maps applied directly to the input in the first affine
transformation of the network.
\end{definition}

We now define shallow topological feedforward neural networks in a form that
corresponds to the standard affine--activation--affine structure of
classical single hidden layer networks.

\begin{definition}[Single hidden layer TFNN with vector output]
\label{def:singleTFNN}
Fix a topological space $X$, a basic family $\mathcal{A}(X)\subset C(X)$, an
activation function $\sigma:\mathbb{R}\to\mathbb{R}$, and an integer $m\ge1$.
A \emph{topological feedforward neural network} (TFNN) with one hidden layer
and $m$ outputs is any function
\[
H:X\to\mathbb{R}^m
\]
of the form
\[
H(x)=A\,\sigma\!\bigl(T(x)\bigr),
\]
where:
\begin{itemize}
\item $T:X\to\mathbb{R}^r$ is a feature map of the form
\[
T(x)=\bigl(w_1 f_1(x)-\theta_1,\dots,w_r f_r(x)-\theta_r\bigr),
\]
with $r\in\mathbb{N}$, $f_i\in\mathcal{A}(X)$, and
$w_i,\theta_i\in\mathbb{R}$;
\item $A\in\mathbb{R}^{m\times r}$ is the linear output-layer weight matrix;
\item the activation function $\sigma$ acts componentwise on vector-valued
arguments.
\end{itemize}
\end{definition}

When $X=\mathbb{R}^d$ and $\mathcal{A}(X)=\mathcal{L}(\mathbb{R}^d)$, the map
$T:\mathbb{R}^d\to\mathbb{R}^r$ is an affine map. In this case,
Definition~\ref{def:singleTFNN} coincides exactly with the classical shallow
feedforward network representation
\[
\Phi(x)=T_1\circ\sigma\circ T_0(x),
\]
where $T_0$ is affine and $T_1$ is linear.

In the case $m=1$, the network in Definition~\ref{def:singleTFNN} can be
written in the classical scalar form
\[
H(x)=\sum_{i=1}^{r} c_i\,\sigma\!\bigl(w_i f_i(x)-\theta_i\bigr),
\]
where $c_i\in\mathbb{R}$ are the output weights.

More generally, for $m\ge1$ the vector-valued network can be written
componentwise as
\[
H_j(x)=\sum_{i=1}^{r} c_{ji}\,
\sigma\!\bigl(w_i f_i(x)-\theta_i\bigr),
\qquad j=1,\dots,m,
\]
where $A=(c_{ji})\in\mathbb{R}^{m\times r}$ is the output weight matrix.

\medskip
\noindent
\textbf{Deep networks on topological spaces.}
Deep feedforward neural networks are obtained by iterating the
affine--activation construction layer by layer. We adopt the same principle
in the topological setting, preserving the standard compositional structure
used in the classical theory.

\begin{definition}[Deep TFNN with vector output]
\label{def:deepTFNN}
Fix a topological space $X$, a basic family $\mathcal{A}(X)\subset C(X)$, an
activation function $\sigma:\mathbb{R}\to\mathbb{R}$, and integers
$l\ge1$, $m\ge1$.

A \emph{deep topological feedforward neural network} (TFNN) of depth $l$ with
$m$ outputs is any function
\[
H:X\to\mathbb{R}^m
\]
of the form
\[
H
=
T_l\circ \sigma \circ T_{l-1}\circ \sigma \circ \cdots \circ \sigma \circ
T_1 \circ \sigma \circ T_0,
\]
where:
\begin{itemize}
\item $T_0:X\to\mathbb{R}^{k_0}$ is a feature map of the form
\[
T_0(x)=\bigl(w_1 f_1(x)-\theta_1,\dots,w_{k_0} f_{k_0}(x)-\theta_{k_0}\bigr),
\]
with $f_i\in\mathcal{A}(X)$ and $w_i,\theta_i\in\mathbb{R}$;

\item for $j=1,\dots,l-1$, $T_j:\mathbb{R}^{k_{j-1}}\to\mathbb{R}^{k_j}$ is an
affine map
\[
T_j(y)=A_j y - b_j,
\qquad
A_j\in\mathbb{R}^{k_j\times k_{j-1}},\;
b_j\in\mathbb{R}^{k_j};
\]

\item $T_l:\mathbb{R}^{k_{l-1}}\to\mathbb{R}^m$ is a linear output map
\[
T_l(y)=A_l y,
\qquad
A_l\in\mathbb{R}^{m\times k_{l-1}};
\]

\item the activation function $\sigma$ acts componentwise on vector-valued
arguments.
\end{itemize}
\end{definition}

This definition agrees with the standard representation of deep
feedforward neural networks as compositions of affine maps in the hidden
layers, componentwise nonlinearities, and a linear output map. In particular,
when $X=\mathbb{R}^d$ and $\mathcal{A}(X)=\mathcal{L}(\mathbb{R}^d)$, deep
TFNNs reduce precisely to classical deep neural networks of the form
\[
\Phi(x)=T_l\circ \sigma \circ T_{l-1} \circ \cdots \circ \sigma \circ T_0(x),
\]
with each $T_j$ affine for $j=0,\dots,l-1$ and $T_l$ linear.

\medskip

In this paper we are especially interested in \emph{deep narrow} (or
\emph{deep skinny}) TFNNs, meaning deep TFNNs for which there exists a fixed
integer $k\in\mathbb{N}$ such that the width of every hidden layer is at most
$k$, while the depth is allowed to grow arbitrarily. In other words, in the
deep narrow setting we assume that
\[
\sup_{j=0,\dots,l-1} k_j \le k,
\]
where the bound $k$ is independent of the depth $l$.

\medskip
\noindent
\textbf{Aim of the paper.}
The aim of this paper is to develop a universality theory for TFNNs with vector-valued 
outputs whose inputs lie in an arbitrary topological space $X$. More precisely, we seek conditions on the
feature family $\mathcal{A}(X)\subset C(X)$ and on the activation function
$\sigma$ under which, for every compact set $K\subset X$, TFNNs are dense
in the space $C(K;\mathbb{R}^m)$ of continuous $\mathbb{R}^m$-valued functions
on $K$.

We also focus on the deep narrow setting, in which the
width of each hidden layer is uniformly bounded while the depth is allowed to
be arbitrary. Our goal is to identify structural assumptions ensuring that
such bounded-width deep TFNNs retain the universal approximation property. In
this way, we extend the deep narrow universality property known for Euclidean
input spaces to a broad class of non-Euclidean spaces.

\medskip
\noindent
\textbf{Related literature.}
The approximation theory of neural networks is centered around the
universal approximation property (UAP), also referred to as density.
In the classical Euclidean setting, this property asserts that, for suitable
activation functions $\sigma$, neural networks with a single hidden layer
are dense in $C(K)$ for every compact set $K\subset\mathbb{R}^d$. The most
general characterization of this phenomenon was obtained by Leshno, Lin,
Pinkus, and Schocken \cite{Leshno}, who proved that, under the assumption of
continuity, an activation function $\sigma$ has the UAP if and only if it is
not a polynomial. This result demonstrates the expressive power of shallow
feedforward networks for a broad class of activation functions. Note that the theorem in \cite{Leshno} applies to a wider class of
activation functions, including some that are discontinuous on sets of
Lebesgue measure zero. Detailed proofs and further discussion
can be found in \cite{Pet,Pinkus}.

Early work on universal approximation emphasized that achieving density
typically requires networks with an unbounded number of hidden neurons 
(see, e.g., \cite[Chapter~6.4.1]{Good}). In these classical results, width was
not constrained. Some studies \cite{GI2,GI3,Ism} have shown that, for
certain non-explicit but computable activation functions, universal
approximation is possible even with a fixed and minimal number of hidden
neurons. These results highlight that width is not always the primary source
of expressive power.

Universality has also been investigated for deep architectures under explicit
assumptions on the activation function. An early contribution in this
direction is due to Gripenberg \cite{Grip}, who studied deep
feedforward networks with a uniformly bounded number of neurons per hidden
layer. He established universal approximation under the assumption that the
activation function is continuous, nonaffine, and twice continuously
differentiable in a neighborhood of some point $t$ with $\sigma''(t)\neq 0$.
Subsequently, Kidger and Lyons \cite{Kid} relaxed the required
smoothness. They showed that universality for deep and narrow networks still 
holds if the activation function is continuous, nonaffine, and differentiable 
in a neighborhood of at least one point, with the derivative being continuous 
and nonzero at that point. In both works, the universality property is derived
from local differentiability assumptions, without imposing any global
smoothness conditions on the activation function.

In contrast, Johnson \cite{John} obtained a negative result for a broad
class of activation functions. He proved that if the activation function is
uniformly continuous and can be approximated by a sequence of one-to-one
functions, then deep feedforward networks with width at most $d$ in each
hidden layer fail to be universal on $\mathbb{R}^d$, regardless of depth.

For a comprehensive and modern treatment of the mathematical analysis of deep
neural networks, we refer the reader to the recent monograph by Petersen and
Zech \cite{Pet}. See also \cite{Ismailov} for a focused discussion of 
how UAP statements are sometimes misinterpreted in the literature, 
and for recent results on universality under fixed architectural size, 
in particular fixed width and fixed depth.

The UAP for neural networks acting between infinite-dimensional spaces has
also been extensively studied. In \cite{Sun}, the fundamentality of ridge
functions in Banach spaces was established and applied to shallow networks
with sigmoidal activation functions (see also \cite{Light}). Here, a subset
$S$ of a topological vector space $E$ is called \emph{fundamental} if the
linear span of $S$ is dense in $E$. In \cite{Chen2}, 
it was shown that any continuous nonlinear operator
$G:V\to C(K_2)$, defined on a compact set
$V\subset C(K_1)$, can be uniformly approximated on
$V\times K_2$ by shallow feedforward neural networks 
in the following sense: for each $y\in K_2$, the scalar-valued map
$u\mapsto G(u)(y)$ admits uniform approximation by 
expressions depending on finitely many point evaluations of $u$.
In \cite{Lu}, this approach was extended to deep neural networks and led to the
introduction of \emph{Deep Operator Networks} (DeepONets), which are designed
to approximate nonlinear continuous operators acting between spaces of
continuous functions. In this setting, both the inputs and outputs of the
operator are functions, and neural networks are used to learn the mapping
$u \mapsto G(u)$. In \cite{Lant}, DeepONets were further analyzed within an
encoder--decoder framework, and their approximation properties were studied
for operators whose inputs lie in Hilbert spaces.

Neural networks with inputs from special topological vector spaces have
also been considered. The universal approximation property for networks with
inputs from Fr\'echet spaces and outputs in Banach spaces was established in
\cite{Gal1}. This framework was further extended in \cite{Gal2}, where
universal approximation theorems were proved for quasi-Polish input and
output spaces. Universal approximation results for functional input neural
networks defined on weighted spaces, with outputs in Banach spaces, were
obtained in \cite{Cuch}.

An alternative approach to neural network approximation over non-Euclidean
spaces relies on certain transfer principles, through which approximation
properties established on Euclidean spaces extend 
to more general domains. For example, \cite{Krat}
investigated which modifications of the input and output layers preserve the
universal approximation property. The authors showed that, under suitable
conditions on topological spaces $X$ and $Y$ and continuous maps
$\phi : X \to \mathbb{R}^n$ and $\rho : \mathbb{R}^m \to Y$, density of a class
$\mathcal{F}\subset C(\mathbb{R}^n,\mathbb{R}^m)$ implies density of the
composed class $\{\rho \circ f \circ \phi : f\in\mathcal{F}\}$ in $C(X,Y)$ with
respect to uniform convergence on compact sets.

In \cite{Krat2}, universal approximation results are established for
neural operators (NOs) and mixtures of neural operators (MoNOs) acting between
Sobolev spaces. Specifically, any nonlinear continuous operator
$G^+ : H^{s_1} \to H^{s_2}$ can be approximated with arbitrary accuracy in the
$L^2$-norm, uniformly on compact subsets $K \subset H^{s_1}$, by NOs and MoNOs.
In addition, \cite{Krat2} provides quantitative bounds on the depth, width, and
rank of the approximating operators in terms of the compact set $K$ and the
target accuracy.

Recent work has also established universal approximation theorems for
hypercomplex-valued neural networks, including complex-, quaternion-, and
Clifford-valued architectures, as well as for more general vector-valued
networks (V-nets) defined over finite-dimensional algebras (see \cite{Valle} and
references therein).

It should be remarked that, in the shallow setting, universality for scalar-valued networks 
on Banach spaces of continuous functions \cite{Chen2} and on general topological spaces 
\cite{Ism2026b}, as well as for networks with values in finite-dimensional algebras 
\cite{Valle}, has been established for activation functions belonging 
to the Tauber--Wiener class. Moreover, for topological spaces with
the Hahn--Banach extension property, scalar-valued universality can be
characterized by non-polynomiality of the activation function on a
sufficiently small interval \cite{Ism2026a}.

The present work differs from much of the above literature in two essential
aspects. First, in addition to treating vector-valued shallow
networks without width constraints, we investigate deep networks with
a uniform width bound and study whether universality can be retained in this
setting. Second, we consider approximation of vector-valued functions defined
on arbitrary topological input spaces. The only restriction of our main
result is that the output space must be a finite-dimensional Euclidean space.
The techniques developed here do not yet extend to neural networks with
outputs in infinite-dimensional spaces. However, we hope that the results of
this paper will stimulate further exploration of shallow and deep neural networks with
outputs taking values in abstract infinite-dimensional spaces.

\section{Universal approximation on topological spaces}

In this section we formulate and prove universality results for 
topological feedforward neural networks whose inputs range over a general
topological space $X$.

To treat approximation in the vector-valued
setting, we work in the space $C(X;\mathbb{R}^m)$ of continuous functions from
$X$ into $\mathbb{R}^m$. In the sequel, $C(X;\mathbb{R}^m)$ is equipped with
the topology of uniform convergence on compact sets. This topology is induced
by the family of seminorms
\[
\|g\|_{K}
=
\sup_{x\in K}\|g(x)\|_{\mathbb{R}^m},
\]
where $K$ ranges over all compact subsets of $X$ and
$\|\cdot\|_{\mathbb{R}^m}$ denotes the Euclidean norm on $\mathbb{R}^m$
(or, equivalently, any fixed norm on $\mathbb{R}^m$).

A subbasis at the origin for this topology is given by the sets
\[
U(K,r)
=
\left\{
g\in C(X;\mathbb{R}^m): \|g\|_{K}<r
\right\},
\]
where $K\subset X$ is compact and $r>0$.

Thus, in what follows, when we say that a family of functions
$\mathcal{F}$ acting from $X$ into $\mathbb{R}^m$ is dense in
$C(X;\mathbb{R}^m)$, we always mean density with respect to the above topology
of uniform convergence on compact sets. That is, $\mathcal{F}$ is dense in
$C(X;\mathbb{R}^m)$ if, for every
$g\in C(X;\mathbb{R}^m)$, every compact set $K\subset X$, and every $\varepsilon>0$, 
there exists $f\in\mathcal{F}$ such that
\[
\|g-f\|_{K}<\varepsilon.
\]
Equivalently, for every compact set $K\subset X$, the collection of
restrictions
\[
\{\,f|_{K} : f\in\mathcal{F}\,\}
\]
is dense in $C(K;\mathbb{R}^m)$ with respect to the uniform seminorm
$\|\cdot\|_{K}$.

Our objective is to relate approximation on the abstract input space $X$ to
approximation on Euclidean spaces. This requires, on the one hand, suitable
richness assumptions on the feature family
$\mathcal{A}(X)\subset C(X)$, and, on the other hand, approximation properties
of the activation function $\sigma$. In this context we distinguish between
two cases. The first case, which covers shallow TFNNs without width 
constraints, can be treated using a structural density condition on the feature 
family, called the $D$-property. The second case, corresponding to deep but 
narrow networks, requires additional ideas that reduce the approximation problem 
to a finite-dimensional setting and exploit depth to compensate for width constraints.

\subsection{\texorpdfstring{The $D$-property}{The D-property}}

We begin by defining a structural density property of feature families that
allows approximation on $X$ to be reduced to approximation of univariate
functions composed with suitable feature maps.

\begin{definition}[$D$-property]
A family $\mathcal{A}(X)\subset C(X)$ is said to have the \emph{$D$-property}
if the linear span
\begin{equation}
\mathcal{S}
=
\operatorname{span}
\{\,u\circ f : u\in C(\mathbb{R}),\ f\in\mathcal{A}(X)\,\}
\label{eq:Dproperty}
\end{equation}
is dense in $C(X)$ with respect to the topology of uniform convergence on
compact sets.

Equivalently, $\mathcal{A}(X)$ has the $D$-property 
if and only if for every function $g\in C(X)$, every
compact set $K\subset X$, and every
$\varepsilon>0$, there exist an integer $m\in\mathbb{N}$, feature maps
$f_1,\dots,f_m\in\mathcal{A}(X)$, and univariate continuous functions
$u_1,\dots,u_m\in C(\mathbb{R})$ such that
\[
\max_{x\in K}
\left|
g(x)-\sum_{i=1}^m u_i\bigl(f_i(x)\bigr)
\right|
<\varepsilon.
\]
\end{definition}

\begin{remark}
The $D$-property asserts that the feature family $\mathcal{A}(X)$ is
sufficiently rich to approximate arbitrary continuous functions on $X$
through finite superpositions with continuous univariate functions. In
particular, the $D$-property implies a strong separation capability: for any
compact set $K\subset X$ and any two distinct points $x,y\in K$, there exists
a function in the span \eqref{eq:Dproperty} that takes different values at
$x$ and $y$.

In the classical setting $X=\mathbb{R}^d$, the role of $\mathcal{A}(X)$ is
played by the family of linear functionals
$\mathbf{x}\mapsto \mathbf{a}\cdot \mathbf{x}$. In this case, density of finite
superpositions of ridge functions $u(\mathbf{a}\cdot \mathbf{x})$ is a
well-studied topic in approximation theory (see, e.g.,
\cite{Ism,Pin}).

\end{remark}

\subsection{A univariate approximation assumption}

To approximate compositions of the form $u\circ f$ by neural networks, we require
that the activation function be able to approximate arbitrary continuous
functions of a single real variable. The following assumption
captures this requirement.

\begin{definition}[Univariate UAP for $\sigma$]
\label{def:uniUAP}
An activation function $\sigma:\mathbb{R}\to\mathbb{R}$ is said to satisfy the
\emph{univariate universal approximation property} if for every compact
interval $[a,b]\subset\mathbb{R}$ and every $\varepsilon>0$, each function
$u\in C([a,b])$ can be approximated uniformly on $[a,b]$ by finite linear
combinations of shifted and rescaled copies of $\sigma$, that is, by functions
of the form
\[
t\mapsto \sum_{j=1}^N c_j\,\sigma(w_j t-\theta_j),
\qquad c_j,w_j,\theta_j\in\mathbb{R}.
\]
\end{definition}

\begin{remark}
Definition~\ref{def:uniUAP} concerns approximation by shallow (single-hidden-layer)
univariate neural networks. This property is classical and is satisfied by many
commonly used activation functions. In particular, it holds for continuous non-polynomial 
activations \cite{Leshno} and characterizes the Tauber--Wiener class
considered in several papers (see, e.g., \cite{Chen2,Ism2026b,Valle}). 
In the present section, only this univariate
approximation property is required.
\end{remark}

\subsection{Universality of TFNNs in the arbitrary-width case}

We now establish a universality statement for topological feedforward neural
networks with vector-valued outputs. Throughout this subsection, no
restriction is imposed on the width: the number of hidden neurons depends 
both on the target function and on the desired accuracy.

For a topological space $X$, a feature family $\mathcal{A}(X)\subset C(X)$,
and an activation function $\sigma:\mathbb{R}\to\mathbb{R}$, we denote by
$\mathcal{N}_1^{(m)}(\sigma)$ the class of all TFNNs with one hidden layer
and $m$ outputs of the form described in Definition~\ref{def:singleTFNN}.
More generally, for each $l\ge1$, we denote by $\mathcal{N}_l^{(m)}(\sigma)$
the class of all deep TFNNs of depth $l$ with $m$ outputs as in
Definition~\ref{def:deepTFNN}.

\begin{theorem}[Universality from the $D$-property]
\label{thm:main}
Let $X$ be a topological space and let $\mathcal{A}(X)\subset C(X)$ be a
feature family with the $D$-property. Assume that the activation function
$\sigma:\mathbb{R}\to\mathbb{R}$ satisfies the univariate universal
approximation property. Then the TFNN class
$\mathcal{N}_1^{(m)}(\sigma)$ is dense in $C(X;\mathbb{R}^m)$ with respect to the 
topology of uniform convergence on compact sets.
That is, for every function $g\in C(X;\mathbb{R}^m)$, every compact
set $K\subset X$, and every
$\varepsilon>0$, there exists a TFNN
$H\in\mathcal{N}_1^{(m)}(\sigma)$ such that
\[
\|g-H\|_{K}<\varepsilon.
\]
\end{theorem}

\begin{proof}
Fix a vector-valued function
\[
g=(g_1,\dots,g_m)\in C(X;\mathbb{R}^m),
\]
a compact set $K\subset X$ and $\varepsilon>0$. We construct a TFNN $H=(H_1,\dots,H_m)$ satisfying
\[
\|g-H\|_{K}
=
\sup_{x\in K}\|g(x)-H(x)\|_{\mathbb{R}^m}
<\varepsilon.
\]

With the Euclidean norm on $\mathbb{R}^m$,
\[
\|g-H\|_{K}
=
\sup_{x\in K}
\left(
\sum_{k=1}^m |g_k(x)-H_k(x)|^2
\right)^{1/2}
\le
\left(
\sum_{k=1}^m \|g_k-H_k\|_{K}^2
\right)^{1/2},
\]
where for scalar-valued functions we use the seminorm
$\|h\|_{K}=\sup_{x\in K}|h(x)|$.

Accordingly, we proceed by constructing, for each component 
$k=1,\dots,m$, a scalar TFNN $H_k$ such that
\[
\|g_k-H_k\|_{K}<\frac{\varepsilon}{\sqrt{m}}.
\]

Fix $k\in\{1,\dots,m\}$ and consider the scalar function $g_k$. By the $D$-property of $\mathcal{A}(X)$, the
linear span of compositions
\[
\operatorname{span}\{u\circ f:\ u\in C(\mathbb{R}),\ f\in\mathcal{A}(X)\}
\]
is dense in $C(X)$ in the topology of uniform convergence on compact sets.
Hence, we can approximate $g_k$
uniformly on $K$ by a finite sum of such compositions. Concretely,
there exist an integer $M_k\in\mathbb{N}$, feature maps
$f_{k,1},\dots,f_{k,M_k}\in\mathcal{A}(X)$, and continuous functions
$u_{k,1},\dots,u_{k,M_k}\in C(\mathbb{R})$ such that
\begin{equation}
\sup_{x\in K}
\left|
g_k(x)-\sum_{i=1}^{M_k} u_{k,i}\bigl(f_{k,i}(x)\bigr)
\right|
<
\frac{\varepsilon}{2\sqrt{m}}.
\label{eq:s1}
\end{equation}

Fix $i\in\{1,\dots,M_k\}$. Since $f_{k,i}$ is continuous and $K$ is
compact, the image
\[
f_{k,i}(K)\subset\mathbb{R}
\]
is compact. Hence there exists a compact interval $[a_{k,i},b_{k,i}]$
with
\[
f_{k,i}(K)\subset [a_{k,i},b_{k,i}].
\]
(For example, take $a_{k,i}=\min_{x\in K}f_{k,i}(x)$ and
$b_{k,i}=\max_{x\in K}f_{k,i}(x)$.)

Thus, to approximate $u_{k,i}(f_{k,i}(x))$ uniformly on $x\in K$, it is
enough to approximate the scalar function $u_{k,i}$ uniformly on the
interval $[a_{k,i},b_{k,i}]$.

By Definition~\ref{def:uniUAP}, for each $i$ there exist an integer
$N_{k,i}\in\mathbb{N}$ and real parameters
$c_{k,i,j},w_{k,i,j},\theta_{k,i,j}\in\mathbb{R}$ such that
\begin{equation}
\sup_{t\in[a_{k,i},b_{k,i}]}
\left|
u_{k,i}(t)-\sum_{j=1}^{N_{k,i}}
c_{k,i,j}\,\sigma(w_{k,i,j}t-\theta_{k,i,j})
\right|
<
\frac{\varepsilon}{2\sqrt{m}\,M_k}.
\label{eq:s3}
\end{equation}
Because $f_{k,i}(K)\subset[a_{k,i},b_{k,i}]$, the same estimate holds
when we substitute $t=f_{k,i}(x)$ with $x\in K$.
Substituting $t=f_{k,i}(x)$ in \eqref{eq:s3} gives
\[
\sup_{x\in K}
\left|
u_{k,i}\bigl(f_{k,i}(x)\bigr)
-
\sum_{j=1}^{N_{k,i}}
c_{k,i,j}\,\sigma\bigl(w_{k,i,j}f_{k,i}(x)-\theta_{k,i,j}\bigr)
\right|
<
\frac{\varepsilon}{2\sqrt{m}\,M_k}.
\]
Now sum over $i=1,\dots,M_k$ and use the triangle inequality:
\begin{align}
&
\sup_{x\in K}
\left|
\sum_{i=1}^{M_k} u_{k,i}\bigl(f_{k,i}(x)\bigr)
-
\sum_{i=1}^{M_k}\sum_{j=1}^{N_{k,i}}
c_{k,i,j}\,\sigma\bigl(w_{k,i,j}f_{k,i}(x)-\theta_{k,i,j}\bigr)
\right|
\nonumber\\
&\qquad\le
\sum_{i=1}^{M_k}
\sup_{x\in K}
\left|
u_{k,i}\bigl(f_{k,i}(x)\bigr)
-
\sum_{j=1}^{N_{k,i}}
c_{k,i,j}\,\sigma\bigl(w_{k,i,j}f_{k,i}(x)-\theta_{k,i,j}\bigr)
\right|
<
\frac{\varepsilon}{2\sqrt{m}}.
\label{eq:s4}
\end{align}
Thus, the finite sum of compositions in \eqref{eq:s1} is uniformly
approximated on $K$ by a finite linear combination of terms of the
form $\sigma(w f(x)-\theta)$.

Define the scalar function
\[
H_k(x)
=
\sum_{i=1}^{M_k}\sum_{j=1}^{N_{k,i}}
c_{k,i,j}\,\sigma\bigl(w_{k,i,j}f_{k,i}(x)-\theta_{k,i,j}\bigr),
\qquad x\in X.
\]
Then $H_k \in\mathcal{N}_1^{(1)}(\sigma)$, which is built from
feature maps in $\mathcal{A}(X)$ and activation $\sigma$.

For each $x\in K$, we insert and subtract the intermediate approximation from
\eqref{eq:s1} and use the triangle inequality:
\begin{align*}
|g_k(x)-H_k(x)|
&\le
\left|
g_k(x)-\sum_{i=1}^{M_k} u_{k,i}\bigl(f_{k,i}(x)\bigr)
\right|
+
\left|
\sum_{i=1}^{M_k} u_{k,i}\bigl(f_{k,i}(x)\bigr)-H_k(x)
\right| \\
&<
\frac{\varepsilon}{2\sqrt{m}}+\frac{\varepsilon}{2\sqrt{m}}
=
\frac{\varepsilon}{\sqrt{m}},
\end{align*}
where we used \eqref{eq:s1} for the first term and \eqref{eq:s4} for the
second term. Taking the supremum over $x\in K$ in all three absolute values yields
\[
\|g_k-H_k\|_{K}<\frac{\varepsilon}{\sqrt{m}}.
\]

We repeat the above for each component $k=1,\dots,m$, obtaining scalar
approximants $H_1,\dots,H_m$. Define
\[
H(x)=(H_1(x),\dots,H_m(x)),\qquad x\in X.
\]
By construction, $H$ is a single hidden layer TFNN with $m$ outputs, i.e.
$H\in\mathcal{N}_1^{(m)}(\sigma)$.

Finally, since $\|g_k-H_k\|_{K}<\varepsilon/\sqrt{m}$ for each $k=1,\dots,m$, it follows that
\[
\|g-H\|_{K}
=
\sup_{x\in K}
\left(
\sum_{k=1}^m |g_k(x)-H_k(x)|^2
\right)^{1/2}
\le
\left(
\sum_{k=1}^m \|g_k-H_k\|_{K}^2
\right)^{1/2}
<
\left(
\sum_{k=1}^m \frac{\varepsilon^2}{m}
\right)^{1/2}
=
\varepsilon.
\]
This completes the proof.
\end{proof}

In Theorem~\ref{thm:main}, the target function is taken
from $C(X;\mathbb{R}^m)$ rather than $C(K;\mathbb{R}^m)$ for a fixed compact set
$K\subset X$. This is due to the formulation of the $D$-property as a density
statement in $C(X)$ with respect to the topology of uniform convergence on
compact sets.

In general topological spaces, a function defined on a compact subset $K$
need not admit a continuous extension to the whole space $X$. Therefore,
working with functions in $C(X;\mathbb{R}^m)$ avoids the need for extension
arguments. However, in spaces where continuous extension from closed sets is available,
such as normal topological spaces, one may equivalently formulate the
approximation property for functions in $C(K;\mathbb{R}^m)$.

\begin{corollary}
Let $X$ be a normal topological space and let $\mathcal{A}(X)\subset C(X)$
satisfy the $D$-property. Assume that $\sigma$ satisfies the univariate
universal approximation property. Then for every compact set $K\subset X$,
every function $g\in C(K;\mathbb{R}^m)$, and every $\varepsilon>0$, there exists
$H\in \mathcal{N}_1^{(m)}(\sigma)$ such that
\[
\|g-H\|_{K}<\varepsilon.
\]
\end{corollary}

\begin{proof}
Let $K\subset X$ be compact, let $g=(g_1,\dots,g_m)\in C(K;\mathbb{R}^m)$,
and let $\varepsilon>0$. Since $X$ is normal and $K$ is closed, the Tietze
extension theorem implies that for each $j=1,\dots,m$ there exists
$G_j\in C(X)$ such that $G_j|_K=g_j$. Define
$G=(G_1,\dots,G_m)\in C(X;\mathbb{R}^m)$; then $G|_K=g$.

Since $\mathcal{A}(X)$ has the $D$-property and $\sigma$ satisfies the
univariate universal approximation property, Theorem~\ref{thm:main}
yields a TFNN $H\in \mathcal{N}_1^{(m)}(\sigma)$ such that
\[
\|G-H\|_{K}<\varepsilon.
\]
Because $G=g$ on $K$, it follows that $\|g-H\|_{K}=\|G-H\|_{K}<\varepsilon$.
\end{proof}

\subsection{Locally convex spaces}

We now analyze an important and widely applicable special case: locally convex
topological vector spaces. In this setting there is a natural
choice of feature family, namely the continuous dual space.

Throughout this subsection, all locally convex 
topological vector spaces are assumed to be Hausdorff.

Let $X$ be a locally convex topological vector space over $\mathbb{R}$, and
denote by $X^*$ its continuous dual. Elements of $X^*$ play the role of linear
measurements of the input, generalizing the inner products
$\mathbf{w}\cdot\mathbf{x}$ that appear in classical neural networks on
$\mathbb{R}^d$.

\begin{theorem}[Universality on locally convex spaces]
\label{thm:lcs}
Let $X$ be a locally convex topological vector space and let
$\sigma:\mathbb{R}\to\mathbb{R}$ satisfy the univariate universal approximation
property of Definition~\ref{def:uniUAP}. Then, for every integer $m\ge1$, the TFNN class
$\mathcal{N}_1^{(m)}(\sigma)$ associated with the feature family
$\mathcal{A}(X)=X^*$ is dense in $C(X;\mathbb{R}^m)$ with respect to the 
topology of uniform convergence on compact sets.
\end{theorem}

\begin{proof}
We first verify that the feature family $\mathcal{A}(X)=X^*$ satisfies the
$D$-property. Define
\[
\mathcal{S}
=
\operatorname{span}
\{\,u\circ f : u\in C(\mathbb{R}),\ f\in X^*\,\}
\subset C(X).
\]
We show that $\mathcal{S}$ is dense in $C(X)$ with respect to uniform
convergence on compact sets.

Let $K\subset X$ be an arbitrary compact set. Consider the set
\[
\mathcal{E}
=
\operatorname{span}
\{\,e^{f(x)} : f\in X^*\,\}
\subset C(X).
\]
This set is an algebra, since
\[
e^{f_1(x)}e^{f_2(x)} = e^{(f_1+f_2)(x)}
\quad\text{and}\quad
f_1+f_2\in X^*,
\]
and it contains the constant functions because $e^{0}=1$.

Moreover, $\mathcal{E}$ separates points of $K$. Indeed, since $X$ is locally
convex, the Hahn--Banach continuous extension theorem holds.
It is a consequence of this theorem that for any $x,y\in K$
with $x\neq y$ there exists $f\in X^*$ such that $f(x)\neq f(y)$ (see, e.g.,
\cite[Theorem~3.6]{Rudin}). Consequently,
\[
e^{f(x)}\neq e^{f(y)}.
\]

By the Stone--Weierstrass theorem \cite{Stone}, the restriction
$\mathcal{E}|_{K}$ is dense in $C(K)$ with respect to the uniform norm.
On the other hand, each generator $e^{f(x)}$ belongs to $\mathcal{S}$ (take
$u(t)=e^{t}$), hence $\mathcal{E}\subset \mathcal{S}$ and therefore
$\mathcal{E}|_K \subset \mathcal{S}|_K$. It follows that $\mathcal{S}|_K$ is
dense in $C(K)$. As $K\subset X$ was arbitrary, this proves that
$\mathcal{A}(X)=X^*$ has the $D$-property.

We now apply Theorem~\ref{thm:main}. Let $m\ge1$, let $K\subset X$ be compact,
and let $g\in C(X;\mathbb{R}^m)$ be arbitrary. Since $\mathcal{A}(X)=X^*$ has
the $D$-property and $\sigma$ satisfies the univariate universal approximation
property, Theorem~\ref{thm:main} implies that for every $\varepsilon>0$ there
exists a TFNN
\[
H\in\mathcal{N}_1^{(m)}(\sigma)
\]
such that
\[
\|g-H\|_{K}<\varepsilon.
\]
The theorem is proved.
\end{proof}

A classical and influential result in this direction is due to Chen and
Chen \cite[Theorem~4]{Chen2}, who studied approximation of continuous
functionals by shallow neural networks on a compact subset $V \subset C(Y)$. 
Their theorem shows that, if the activation function belongs to the
Tauber--Wiener class, then every continuous functional on $V$ can be
uniformly approximated by shallow neural network functionals whose inputs
depend only on finitely many point evaluations
$u(x_1),\dots,u(x_k)$, where $u\in V$.

In the present framework, a corresponding approximation result can be obtained
as a consequence of the universality theorem for locally convex spaces
(Theorem~\ref{thm:lcs}), together with the fact that, on compact subsets of
$C(Y)$, continuous linear functionals can be uniformly approximated by finite
linear combinations of point-evaluation functionals. The following theorem
makes this connection explicit for continuous activation functions.

\begin{theorem}[Chen and Chen \cite{Chen2}]
\label{thm:chen-chen}
Let $Y$ be a compact metric space and let $V\subset C(Y)$ be compact. 
Assume that the activation $\sigma \in C(\mathbb{R})$
satisfies the univariate universal approximation property of
Definition~\ref{def:uniUAP}. Then for every continuous functional $f\in C(V)$
and every $\varepsilon>0$ there exist integers $N,k\ge1$, points
$x_1,\dots,x_k\in Y$, and real parameters $c_i,\theta_i,\xi_{ij}$
($i=1,\dots,N$, $j=1,\dots,k$) such that
\[
\left|
f(u)-\sum_{i=1}^{N}c_i\,
\sigma\!\left(\sum_{j=1}^{k}\xi_{ij}\,u(x_j)-\theta_i\right)
\right|
<\varepsilon
\]
holds for all $u\in V$.
\end{theorem}

\begin{proof}
We view $C(Y)$ as a Banach space and hence as a Hausdorff locally convex space.
Since $C(Y)$ is a metric space, it is normal, and the compact set $V\subset C(Y)$
is closed. Therefore, by the Tietze extension theorem, the function
$f\in C(V)$ admits a continuous extension $F\in C(C(Y))$.
Applying Theorem~\ref{thm:lcs} to $F$ with $X=C(Y)$ and $m=1$, we obtain
$N\ge1$, functionals $\ell_1,\dots,\ell_N\in C(Y)^*$, and parameters
$c_i,\theta_i\in\mathbb{R}$ such that
\[
\sup_{u\in V}\left|
F(u)-\sum_{i=1}^{N} c_i\,\sigma(\ell_i(u)-\theta_i)
\right|
<\frac{\varepsilon}{2}.
\]
Since $F=f$ on $V$, it follows that
\begin{equation}
\sup_{u\in V}\left|
f(u)-\sum_{i=1}^{N} c_i\,\sigma(\ell_i(u)-\theta_i)
\right|
<\frac{\varepsilon}{2}.
\label{eq:sA}
\end{equation}

Next we approximate each $\ell_i$ uniformly on $V$ by finitely many point
evaluations. Since $V\subset C(Y)$ is compact, it is equicontinuous by the
Arzelà--Ascoli theorem. Let $\ell\in C(Y)^*$ be arbitrary. By the Riesz
representation theorem, $\ell(u)=\int_Y u\,d\mu$ for a finite signed Borel
measure $\mu$ on $Y$. Equicontinuity of $V$ implies that for
every $\delta>0$ there exist points $x_1,\dots,x_k\in Y$ and coefficients
$\alpha_1,\dots,\alpha_k\in\mathbb{R}$ such that
\begin{equation}
\sup_{u\in V}\left|
\ell(u)-\sum_{j=1}^{k}\alpha_j\,u(x_j)
\right|
<\delta.
\label{eq:sB}
\end{equation}

Now apply \eqref{eq:sB} to each
$\ell_i$, $i=1,\dots,N$. For each $i$ this yields a finite set of sample points
$\{x_{i,1},\dots,x_{i,k_i}\}\subset Y$ and coefficients $\alpha_{i,j}$ such that
\[
\sup_{u\in V}\left|
\ell_i(u)-\sum_{j=1}^{k_i}\alpha_{i,j}\,u(x_{i,j})
\right|
<\delta.
\]
Let $\{x_1,\dots,x_k\}$ be the union of all these sampling points over
$i=1,\dots,N$. By setting the coefficients to zero if necessary, each sum may be rewritten
over this common set, yielding coefficients $\xi_{ij}$ such that
\[
\sup_{u\in V}\left|
\ell_i(u)-\sum_{j=1}^{k}\xi_{ij}\,u(x_j)
\right|
<\delta,
\qquad i=1,\dots,N.
\]

Since $V$ is compact and each $\ell_i$ is continuous, the set
$\ell_i(V)\subset\mathbb{R}$ is compact. Therefore $\sigma$ is uniformly
continuous on a compact interval containing both $\ell_i(V)$ and the
approximants $\sum_j\xi_{ij}u(x_j)$ for $u\in V$ (for $\delta$ small enough).
Choose $\delta>0$ so that
\[
|s-t|<\delta
\quad\Longrightarrow\quad
|\sigma(s-\theta_i)-\sigma(t-\theta_i)|
<
\frac{\varepsilon}{2\,\sum_{i=1}^{N}|c_i|+1}
\qquad (i=1,\dots,N).
\]
Then for every $u\in V$,
\[
\left|
c_i\sigma(\ell_i(u)-\theta_i)
-
c_i\sigma\!\left(\sum_{j=1}^{k}\xi_{ij}u(x_j)-\theta_i\right)
\right|
\le
|c_i|\cdot
\frac{\varepsilon}{2\,\sum_{i=1}^{N}|c_i|+1}.
\]
Summing over $i=1,\dots,N$ gives
\begin{equation}
\sup_{u\in V}\left|
\sum_{i=1}^{N} c_i\,\sigma(\ell_i(u)-\theta_i)
-
\sum_{i=1}^{N} c_i\,\sigma\!\left(\sum_{j=1}^{k}\xi_{ij}u(x_j)-\theta_i\right)
\right|
<
\frac{\varepsilon}{2}.
\label{eq:sC}
\end{equation}

Finally, combine \eqref{eq:sA} and \eqref{eq:sC} by the triangle inequality. 
This yields the desired representation with error $<\varepsilon$.
\end{proof}

\begin{remark}
Theorem~\ref{thm:lcs} shows that locally convex spaces (including Banach spaces and
Fr\'echet spaces) form a broad and natural class of input domains for which
universality of topological feedforward neural networks holds. 
This universality result is based on the abstract
$D$-property and the separating power of the continuous dual. 
The subsequent Theorem~\ref{thm:chen-chen} of Chen and Chen may be viewed as
a special case of this general principle applied to the Banach space $C(Y)$. 
In the next section, we turn to the substantially more delicate question of
whether universality can be retained under uniform width constraints, that is,
in the deep narrow framework.
\end{remark}

\section{Universality of deep narrow TFNNs}

In this section we study whether universal approximation on a general
topological space $X$ can be achieved by deep neural networks with uniformly
bounded width. To address this question, it is not sufficient to rely solely on
the $D$-property. While this property ensures approximation by finite sums of
univariate compositions, it does not provide control over the number of neurons
required in each hidden layer.

To overcome this limitation, we impose a \emph{finite-dimensional
composition} assumption on the feature family $\mathcal{A}(X)$. Under this 
assumption, approximation on $X$ can be reduced to approximation on a compact 
subset of $\mathbb{R}^n$, where known deep narrow universality results, 
such as those of Gripenberg \cite{Grip} and Kidger and Lyons \cite{Kid}, apply.

\subsection{Finite-dimensional composition via feature maps}

We begin by formalizing the composition property. The key point, compared with
the $D$-property, is that for each compact set $K\subset X$ and each output
dimension $m$ one fixes a finite collection of feature maps, and approximation
of different target functions is achieved solely by varying the outer function.

\begin{definition}[Finite-dimensional composition]
\label{def:fd-density}
A feature family $\mathcal{A}(X)\subset C(X)$ is said to have the
\emph{finite-dimensional composition property of order $n$}
if for every compact set $K\subset X$ and every integer $m\ge1$ there exist
functions $f_1,\dots,f_n\in\mathcal{A}(X)$ such that, with
$F=(f_1,\dots,f_n):X\to\mathbb{R}^n$, the set
\[
\{\,u\circ F|_{K} : u\in C(\mathbb{R}^n;\mathbb{R}^m)\,\}
\]
is dense in $C(K;\mathbb{R}^m)$.
\end{definition}

\begin{remark}
Definition~\ref{def:fd-density} means that, on each compact subset $K$ of $X$,
every continuous $\mathbb{R}^m$-valued function can be approximated arbitrarily
well by first mapping $x\in X$ to the Euclidean feature vector
$F(x)\in\mathbb{R}^n$ and then applying a continuous map
$u:\mathbb{R}^n\to\mathbb{R}^m$. In contrast to the $D$-property, the same
finite set of features $f_1,\dots,f_n$ is used for all target functions and
all accuracies; only the mapping $u$ is allowed to change.
\end{remark}

The finite-dimensional composition property of
Definition~\ref{def:fd-density} is formulated as a density condition.
In many situations, however, this property holds in a stronger form, namely
with exact representations rather than mere approximation.
The following proposition identifies situations in which 
such a strong finite-dimensional composition property holds on compact sets.

\begin{proposition}[Strong composition property]
\label{prop:strong-fd-composition}
Let $X$ be a topological space and let $K\subset X$ be compact.

\smallskip
\noindent
\textnormal{(i)}
Suppose that there exists a feature map
\[
F=(f_1,\dots,f_n):X\to\mathbb{R}^n,
\qquad f_i\in\mathcal{A}(X),
\]
such that the restriction $F|_K:K\to F(K)$ is a topological embedding. Then
the finite-dimensional composition property on $K$ holds in a strong form: for
every integer $m\ge1$ and every function $g\in C(K;\mathbb{R}^m)$ there exists
a function $u\in C(\mathbb{R}^n;\mathbb{R}^m)$ such that
\[
g = u\circ F|_K .
\]

\smallskip
\noindent
\textnormal{(ii)}
Suppose that there exists a topological embedding
\[
E=(E_1,\dots,E_n):K\hookrightarrow \mathbb{R}^n
\]
and functions $f_1,\dots,f_n\in\mathcal{A}(X)$ such that
\[
f_i|_K = E_i,\qquad i=1,\dots,n.
\]
Let $F=(f_1,\dots,f_n):X\to\mathbb{R}^n$. Then the same strong finite-dimensional composition property
on $K$ holds:
for every $m\ge1$ and every $g\in C(K;\mathbb{R}^m)$ there exists
$u\in C(\mathbb{R}^n;\mathbb{R}^m)$ such that
\[
g = u\circ F|_K .
\]
\end{proposition}

\begin{proof}
\textnormal{(i)}
Since $K$ is compact and $F$ is continuous, the set $F(K)\subset\mathbb{R}^n$
is compact and hence closed. As $F|_K:K\to F(K)$ is a topological embedding, it
is a homeomorphism between $K$ and $F(K)$. Given any $g\in C(K;\mathbb{R}^m)$, define a function
$v:F(K)\to\mathbb{R}^m$ by
\[
v(y)=g\bigl((F|_K)^{-1}(y)\bigr), \qquad y\in F(K).
\]
Then $v$ is continuous on the compact set $F(K)$. By the Tietze extension
theorem \cite[Theorem~15.8]{Wil} (applied componentwise), there exists a function
$u\in C(\mathbb{R}^n;\mathbb{R}^m)$ such that $u|_{F(K)}=v$. Consequently,
for all $x\in K$,
\[
(u\circ F)(x)=u(F(x))=v(F(x))=g(x),
\]
and hence $g=u\circ F|_K$.

\smallskip
\textnormal{(ii)}
By assumption, $E:K\to\mathbb{R}^n$ is a topological embedding and
\[
F|_K=(f_1|_K,\dots,f_n|_K)=(E_1,\dots,E_n)=E.
\]
Hence $F|_K$ is a topological
embedding as well, and the conclusion follows directly from part~(i).
\end{proof}

\begin{remark}
Proposition~\ref{prop:strong-fd-composition} separates two
distinct issues. Part~(i) shows that, once a finite-dimensional feature map
$F=(f_1,\dots,f_n)$ with $f_i\in\mathcal{A}(X)$ is available whose restriction
to a compact set $K$ is a topological embedding, the finite-dimensional
composition property on $K$ holds in a strong form, in fact yielding exact
representations $g=u\circ F|_K$ for all $g\in C(K;\mathbb{R}^m)$.

Part~(ii) clarifies how such feature maps may arise in practice. In many
settings, the existence of an embedding $E:K\hookrightarrow\mathbb{R}^n$ is a
purely topological fact, independent of the choice of feature family
$\mathcal{A}(X)$ (for example, by classical embedding theorems for compact
metric spaces). The additional requirement in part~(ii) is that the coordinate
functions of such an embedding can be realized by elements of
$\mathcal{A}(X)$. In this case, the embedding $E$ can be extended to a feature map
$F$ on $X$, and the strong finite-dimensional composition property on $K$
follows.
\end{remark}

\subsection{Deep narrow TFNNs}

We now consider the class of bounded-width deep networks 
that will be used in the sequel.

Fix $n\in\mathbb{N}$ and a feature map $F:X\to\mathbb{R}^n$ of the form
$F=(f_1,\dots,f_n)$ with $f_i\in\mathcal{A}(X)$. Any (classical) deep
feedforward network $\Phi:\mathbb{R}^n\to\mathbb{R}^m$ with activation
$\sigma$ then induces a vector-valued network on $X$ by composition:
\[
H(x)=\Phi(F(x)),\qquad x\in X.
\]
Intuitively, $F$ provides the input coordinates, and $\Phi$ performs the
standard Euclidean deep processing. We use the standard notion of width for 
Euclidean deep networks: the width of a network is the maximum 
number of neurons in any hidden layer of the Euclidean network $\Phi$.

\begin{definition}[Deep narrow TFNNs based on $F$]
\label{def:widthk}
Let $k\in\mathbb{N}$ and $m\ge1$. We denote by $\mathcal{N}^{(k,m)}(X,\sigma;F)$
the class of all functions $H:X\to\mathbb{R}^m$ for which there exists a
(classical) deep feedforward network $\Phi:\mathbb{R}^n\to\mathbb{R}^m$ with
activation function $\sigma$ and with width at most $k$ such that
\[
H(x)=\Phi(F(x)),\qquad x\in X.
\]
We refer to the elements of $\mathcal{N}^{(k,m)}(X,\sigma;F)$ as
\emph{deep narrow TFNNs with $m$ outputs based on $F$}.
\end{definition}

\subsection{Universality theorem under a width bound}

We now state the universality theorem for deep narrow TFNNs.

\begin{theorem}[Universality of deep narrow TFNNs]
\label{thm:skinny}
Let $X$ be a topological space and let $\mathcal{A}(X)\subset C(X)$ satisfy
the finite-dimensional composition property of order $n$
(Definition~\ref{def:fd-density}). Assume that $\sigma:\mathbb{R}\to\mathbb{R}$
is continuous and nonaffine, and that there exists a point $t_0\in\mathbb{R}$
and an open neighborhood $U$ of $t_0$ such that $\sigma$ is differentiable on
$U$ and $\sigma'$ is continuous and nonzero at $t_0$.

Then the following holds: for every integer $m\ge1$ and every compact set
$K\subset X$, if $f_1,\dots,f_n\in\mathcal{A}(X)$ are feature maps provided by
Definition~\ref{def:fd-density} for this $K$ and $m$, and
$F=(f_1,\dots,f_n):X\to\mathbb{R}^n$, then for every function
$g\in C(K;\mathbb{R}^m)$ and every $\varepsilon>0$, there exists a function
\[
H\in \mathcal{N}^{(n+m+2,m)}(X,\sigma;F)
\]
such that
\[
\|g-H\|_{K}<\varepsilon.
\]
That is, the class $\mathcal{N}^{(n+m+2,m)}(X,\sigma;F)|_{K}$ is dense in
$C(K;\mathbb{R}^m)$.
\end{theorem}

\begin{proof}
Fix $m\ge1$, a compact set $K\subset X$, a target function
$g\in C(K;\mathbb{R}^m)$, and $\varepsilon>0$.

By the finite-dimensional composition property
(Definition~\ref{def:fd-density}) applied to this particular compact set $K$
and output dimension $m$, there exist feature maps
$f_1,\dots,f_n\in\mathcal{A}(X)$ and a continuous map
$u\in C(\mathbb{R}^n;\mathbb{R}^m)$ such that, writing
\[
F=(f_1,\dots,f_n):X\to\mathbb{R}^n,
\]
we have
\begin{equation}
\|g-u\circ F\|_{K}<\frac{\varepsilon}{2}.
\label{eq:st1}
\end{equation}
Here the feature map $F$ depends on $K$ and $m$, while the map $u$ additionally depends
on $g$ and $\varepsilon$.

Since $F$ is continuous and $K$ is compact, the image
\[
K_F:=F(K)\subset\mathbb{R}^n
\]
is compact. Consider the restriction $u|_{K_F}\in C(K_F;\mathbb{R}^m)$.

Under the stated regularity assumptions on $\sigma$, the deep narrow
universality theorem of Kidger and Lyons \cite[Theorem~3.2]{Kid} yields the
following conclusion: for the compact set $K_F\subset\mathbb{R}^n$ and
tolerance $\varepsilon/2$, there exists a classical deep feedforward network
\[
\Phi:\mathbb{R}^n\to\mathbb{R}^m
\]
with activation function $\sigma$ and with width at most $n+m+2$ such that
\begin{equation}
\max_{y\in K_F}\|u(y)-\Phi(y)\|_{\mathbb{R}^m}<\frac{\varepsilon}{2}.
\label{eq:st2}
\end{equation}

Define
\[
H(x)=\Phi(F(x)),\qquad x\in X.
\]
By Definition~\ref{def:widthk}, since $\Phi$ has width at most $n+m+2$, we have
\[
H\in \mathcal{N}^{(n+m+2,m)}(X,\sigma;F).
\]

Let $x\in K$. Then $F(x)\in K_F$, and therefore \eqref{eq:st2} implies
\[
\|u(F(x))-\Phi(F(x))\|_{\mathbb{R}^m}<\frac{\varepsilon}{2}.
\]
Together with \eqref{eq:st1}, we obtain
\[
\|g(x)-H(x)\|_{\mathbb{R}^m}
\le
\|g(x)-u(F(x))\|_{\mathbb{R}^m}
+
\|u(F(x))-\Phi(F(x))\|_{\mathbb{R}^m}
<
\frac{\varepsilon}{2}+\frac{\varepsilon}{2}
=
\varepsilon.
\]
Taking the maximum over $x\in K$ yields $\|g-H\|_{K}<\varepsilon$, which
completes the proof.
\end{proof}

\begin{remark}[Scope and realizability of Theorem~\ref{thm:skinny}]
The finite-dimensional composition property of order $n$ imposes a
structural restriction on the input space $X$. It requires that for
every compact set $K\subset X$ and every output dimension $m$, there exist
continuous feature maps
\[
F=(f_1,\dots,f_n):X\to\mathbb{R}^n
\]
such that every continuous $\mathbb{R}^m$-valued function on $K$ can be
uniformly approximated by functions of the form $u\circ F$ with
$u\in C(\mathbb{R}^n;\mathbb{R}^m)$.

In particular, for each compact set $K\subset X$, the restriction $F|_K$ must
be injective. Indeed, if $x,y\in K$ with $x\neq y$ and $F(x)=F(y)$, then
\[
(u\circ F)(x)=(u\circ F)(y)
\quad\text{for every } u\in C(\mathbb{R}^n;\mathbb{R}^m).
\]
Assume for the moment that $X$ is Hausdorff. Then $K$ is compact Hausdorff and
hence normal. Since $\{x\}$ and $\{y\}$ are disjoint closed subsets of $K$,
Urysohn's lemma \cite[Theorem~15.6]{Wil} yields a function $h\in C(K)$ such that $h(x)\neq h(y)$.
Consequently,
$g=(h,0,\dots,0)\in C(K;\mathbb{R}^m)$
satisfies $g(x)\neq g(y)$. Such a function cannot be approximated uniformly on
$K$ by functions of the form $u\circ F$, contradicting
Definition~\ref{def:fd-density}. Hence $F|_K$ is injective.

We see that if $X$ is Hausdorff, then $F|_K$ is a continuous injective map from the compact
space $K$ into the Hausdorff space $\mathbb{R}^n$, and therefore a topological
embedding of $K$ into $\mathbb{R}^n$ (a homeomorphism onto its image $F(K)$ \cite[Theorem~3.1.13]{Engel2}).
On the other hand, whenever $F|_K$ is a topological embedding,
Proposition~\ref{prop:strong-fd-composition}(i) implies that
every function $g\in C(K;\mathbb{R}^m)$ admits an exact representation
\[
g = u\circ F|_K
\qquad\text{for some } u\in C(\mathbb{R}^n;\mathbb{R}^m).
\]
It follows that, when $X$ is Hausdorff, the finite-dimensional composition
property on a compact set $K\subset X$ automatically implies its strong form
on $K$.

In the class of separable metric spaces, the above composition property is closely
related to finiteness of the topological (covering) dimension. More precisely,
if $X$ is a separable metric space with $\dim_{\mathrm{top}} X \le d$, then by
the Menger--N\"obeling embedding theorem \cite[Theorem~1.11.4]{Engel} 
there exists a topological embedding
\[
E:X\hookrightarrow\mathbb{R}^{2d+1}.
\]
If, in addition, the feature family $\mathcal{A}(X)$ is sufficiently rich to
contain the coordinate functions of such an embedding, that is, if
\[
E=(E_1,\dots,E_{2d+1}) \quad\text{with}\quad E_i\in\mathcal{A}(X),
\]
then Definition~\ref{def:fd-density} is satisfied with $n=2d+1$ by taking
$f_i=E_i$ and $F=E$. In this case, the same fixed feature map $F$ embeds every
compact subset of $X$ into $\mathbb{R}^{2d+1}$, and Proposition~\ref{prop:strong-fd-composition}(ii)
yields exact representations on each compact set.

As a special case, when $X=\mathbb{R}^n$ and $\mathcal{A}(X)$ consists of the
coordinate projections, the feature map $F$ coincides with the identity on
$\mathbb{R}^n$. In this setting, Theorem~\ref{thm:skinny} reduces exactly to the
deep narrow universality theorem of Kidger and Lyons for Euclidean input spaces
and vector-valued outputs.
\end{remark}

\subsection{Deep narrow universality based on Ostrand inner functions}

We begin by recalling the classical Kolmogorov superposition theorem (KST)
and its extension due to Ostrand. For background, refinements, and
generalizations of KST we refer to \cite[Chapter~1]{Kh} and
\cite[Chapter~4]{Ism}. Kolmogorov's original theorem \cite{Kol} asserts that for the unit
cube $\mathbb{I}^d$, $\mathbb{I}=[0,1]$, $d\ge2$, there exist continuous
functions $s_q\in C(\mathbb{I}^d)$, $q=1,\dots,2d+1$, of the form
\[
s_q(x_1,\dots,x_d)=\sum_{p=1}^d \varphi_{pq}(x_p),
\qquad \varphi_{pq}\in C(\mathbb{I}),
\]
such that every $f\in C(\mathbb{I}^d)$ admits a representation
\[
f(x)=\sum_{q=1}^{2d+1} g_q\bigl(s_q(x)\bigr),
\qquad g_q\in C(\mathbb{R}).
\]
This deep result, which solved Hilbert's 13th problem in the negative, has
stimulated extensive research and has also been discussed at length in
connection with neural networks (see, e.g., \cite{AV} and the references
therein).

Ostrand \cite{Ost} extended KST to general compact metric spaces as follows.

\begin{theorem}[Ostrand {\cite{Ost}}]
\label{thm:ostrand}
For $p=1,\dots,n_0$, let $X_p$ be a compact metric space of finite
(topological) dimension $d_p$, and set $M=\sum_{p=1}^{n_0} d_p$.
There exist universal continuous functions
$\psi_{pq}:X_p\to[0,1]$, $p=1,\dots,n_0$, $q=1,\dots,2M+1$, such that every
$f\in C\!\left(\prod_{p=1}^{n_0} X_p\right)$ is representable in the form
\[
f(x_1,\dots,x_{n_0})
=
\sum_{q=1}^{2M+1}
h_q\!\left(\sum_{p=1}^{n_0}\psi_{pq}(x_p)\right),
\qquad h_q\in C(\mathbb{R}).
\]
\end{theorem}

Define the associated \emph{Ostrand inner
functions} on $X=\prod_{p=1}^{n_0}X_p$ by
\begin{equation}
s_q(x_1,\dots,x_{n_0})
:=
\sum_{p=1}^{n_0}\psi_{pq}(x_p),
\qquad q=1,\dots,2M+1.
\label{eq:ostrand-sq}
\end{equation}
Theorem~\ref{thm:ostrand} states that every scalar-valued continuous function
on $X$ is a finite sum of univariate superpositions $h_q\circ s_q$.
In particular, if the functions $s_q$ are available as admissible features,
then they yield a finite-dimensional composition in the sense of
Definition~\ref{def:fd-density}. Indeed, letting
\[
F=(s_1,\dots,s_{2M+1}):X\to\mathbb{R}^{2M+1},
\]
the exact superposition formula implies that every
$g\in C(K;\mathbb{R}^m)$ admits an exact representation of the form
$g=u\circ F|_K$ for some
$u\in C(\mathbb{R}^{2M+1};\mathbb{R}^m)$.
This is precisely the strong finite-dimensional composition property identified
in Proposition~\ref{prop:strong-fd-composition}.

\begin{theorem}[Deep narrow universality from Ostrand-type features]
Let $X=\prod_{p=1}^{n_0} X_p$ be a product of compact metric spaces, with
$\dim_{\mathrm{top}}(X_p)=d_p<\infty$, and set $M=\sum_{p=1}^{n_0} d_p$.
Let $\mathcal{A}(X)\subset C(X)$ be a feature family and assume that the
Ostrand inner functions $s_q$ defined in \eqref{eq:ostrand-sq} belong to
$\mathcal{A}(X)$ for $q=1,\dots,2M+1$. Let $F=(s_1,\dots,s_{2M+1})$ as above.
Assume that $\sigma:\mathbb{R}\to\mathbb{R}$ is continuous and nonaffine, and
that there exists $t_0\in\mathbb{R}$ and an open neighborhood $U$ of $t_0$
such that $\sigma$ is differentiable on $U$ and $\sigma'$ is continuous and
nonzero at $t_0$.

Then for every integer $m\ge1$, every compact set $K\subset X$, every function
$g\in C(K;\mathbb{R}^m)$, and every $\varepsilon>0$, there exists a function
\[
H\in \mathcal{N}^{(2M+m+3,m)}(X,\sigma;F)
\]
such that
\[
\|g-H\|_{K}<\varepsilon.
\]
Equivalently, for each fixed $m$ and compact $K\subset X$, the class
$\mathcal{N}^{(2M+m+3,m)}(X,\sigma;F)|_{K}$ is dense in $C(K;\mathbb{R}^m)$.
\end{theorem}

\begin{proof}
Fix $m\ge1$, a compact set $K\subset X$, a target function
$g=(g_1,\dots,g_m)\in C(K;\mathbb{R}^m)$, and $\varepsilon>0$.

By Theorem~\ref{thm:ostrand}, for each component $g_r\in C(K)$ ($r=1,\dots,m$)
there exist functions $h_{r,q}\in C(\mathbb{R})$, $q=1,\dots,2M+1$, such that
for all $x\in K$,
\[
g_r(x)=\sum_{q=1}^{2M+1} h_{r,q}\bigl(s_q(x)\bigr).
\]
Define $u:\mathbb{R}^{2M+1}\to\mathbb{R}^m$ by
\[
u(y_1,\dots,y_{2M+1})
:=
\left(
\sum_{q=1}^{2M+1} h_{1,q}(y_q),
\ \dots,\
\sum_{q=1}^{2M+1} h_{m,q}(y_q)
\right).
\]
Then $u\in C(\mathbb{R}^{2M+1};\mathbb{R}^m)$ and, by construction,
\[
g(x)=u(F(x))
\qquad\text{for all } x\in K.
\]
Thus, $g$ admits an exact representation of the form
$g=u\circ F|_K$, which is the strong finite-dimensional composition property
from Proposition~\ref{prop:strong-fd-composition}(i).

Since $F$ is continuous and $K$ is compact, the image $K_F:=F(K)$ is a compact
subset of $\mathbb{R}^{2M+1}$. By the deep narrow universality theorem of
Kidger and Lyons \cite[Theorem~3.2]{Kid}, applied to the compact set $K_F$ and
to the continuous map $u|_{K_F}\in C(K_F;\mathbb{R}^m)$, there exists a
(classical) deep feedforward network
\[
\Phi:\mathbb{R}^{2M+1}\to\mathbb{R}^m
\]
with activation $\sigma$ and width at most $(2M+1)+m+2$ such that
\[
\max_{y\in K_F}\|u(y)-\Phi(y)\|_{\mathbb{R}^m}<\varepsilon.
\]

Define $H=\Phi\circ F$. Then, by Definition~\ref{def:widthk},
\[
H\in \mathcal{N}^{(2M+m+3,m)}(X,\sigma;F).
\]
Moreover, for $x\in K$ we have $F(x)\in K_F$, hence
\[
\|g(x)-H(x)\|_{\mathbb{R}^m}
=
\|u(F(x))-\Phi(F(x))\|_{\mathbb{R}^m}
<\varepsilon.
\]
Taking the maximum over $x\in K$ yields $\|g-H\|_{K}<\varepsilon$.
\end{proof}

\begin{remark}
In \cite{Ism2026b} we introduced topological feedforward neural
networks and proved a universal approximation theorem for single-hidden-layer 
networks with scalar outputs on general topological spaces. As an application, we
derived an approximative version of the Kolmogorov superposition theorem for
compact metric spaces by replacing the non-fixed outer functions with a
single infinitely smooth activation function.

The present paper extends this viewpoint in two directions. First, we treat
vector-valued outputs and consider both shallow and deep architectures. Second, we address 
the bounded-width setting: deep narrow universality is
established under a finite-dimensional composition property, which identifies
when approximation on a general topological space can be reduced to
approximation on a Euclidean space.
In particular, for products of compact metric spaces we obtain explicit width
bounds expressed in terms of topological dimension, thereby linking
geometric and topological structure to quantitative architectural
constraints.
\end{remark}

\begin{remark}
Related results on approximation of mappings between general topological
spaces are developed in \cite{Krat}. 
In that work, one starts from a neural network class
$\mathcal{F}\subset C(\mathbb{R}^n,\mathbb{R}^m)$ that is already known to be
dense and investigates when density is preserved under modifications of the
input and output layers. More precisely, under suitable assumptions on
continuous maps $\phi:X\to\mathbb{R}^n$ and $\rho:\mathbb{R}^m\to Y$, in
particular assuming that $\phi$ is injective, the authors obtain density of the
composed class
\[
\{\,\rho\circ f\circ \phi : f\in\mathcal{F}\,\}
\]
in $C(X,Y)$ with respect to the topology of uniform convergence on compact
sets.

Our deep narrow construction is related in spirit to composition-based
transfer principles when $Y=\mathbb{R}^m$ and $\rho=\mathrm{id}$, since
networks on $X$ are represented as compositions $H=\Phi\circ F$, where
$\Phi$ is a Euclidean neural network and $F:X\to\mathbb{R}^n$ is a feature map. 
However, the role played by
$F$ in the present work differs fundamentally from the role of the fixed inner
map $\phi$ in \cite{Krat}. In \cite{Krat}, the map $\phi$ is assumed a priori
and serves as a fixed, globally injective mapping of $X$ into a Euclidean
space, through which approximation properties are transferred from
$\mathbb{R}^n$ to $X$.

In contrast, in our framework the feature map $F$ is intrinsic to the neural
network architecture on $X$ and is constructed from an admissible feature
family $\mathcal{A}(X)\subset C(X)$. In the universality results without width
constraints, the maps $F=(f_1,\dots,f_n)$ entering the first hidden
layer of the network representation depend on the compact set under
consideration, the target function, and the required accuracy, similarly to
trainable parameters in classical networks. In the deep narrow setting, a
finite-dimensional composition is imposed, and $F=(f_1,\dots,f_n)$ depends only
on the compact set $K$ and the output dimension $m$. Thus, in both the
arbitrary-width and bounded-width settings, 
the feature map $F$ is variable rather
than a single fixed global injective map as in \cite{Krat}.

Accordingly, in the present work injectivity of the feature map $F$ is not
imposed \emph{a priori}, but is shown to be a necessary consequence of the
finite-dimensional composition property
(Definition~\ref{def:fd-density}). This property specifies when a finite
collection of admissible features $f_1,\dots,f_n\in\mathcal{A}(X)$ 
suffices to reduce approximation on $K$ to
approximation on the compact subset $F(K)\subset\mathbb{R}^n$. When $F|_K$
is a topological embedding, Proposition~\ref{prop:strong-fd-composition}
shows that this reduction holds in a strong form, yielding exact
representations $g=u\circ F|_K$ for all $g\in C(K;\mathbb{R}^m)$.

This shifts the focus from the existence of an abstract embedding $\phi$ to
concrete structural properties of the feature family $\mathcal{A}(X)$ and
the topology of $X$. In particular, topological assumptions such as the
Hausdorff property play a natural role: if $X$ is not Hausdorff, a feature
map $F:X\to\mathbb{R}^n$ need not be a topological embedding, and the strong
composition property may fail. Hence, our framework is inherently
topology-sensitive rather than merely compositional.

Moreover, in the compact metric setting we give an explicit realizability
mechanism based on classical dimension theory. Ostrand's extension of the
Kolmogorov superposition theorem yields concrete feature maps for products of
compact metric spaces. This leads to deep narrow universality with width
bounds stated in terms of topological dimension. In this way, geometric and
topological structure translates into width-controlled universality on
non-Euclidean spaces.
\end{remark}

\section{Discussion and Outlook}

This paper develops a general approximation-theoretic framework for neural
networks whose inputs range over arbitrary topological spaces. By formulating
neural architectures in terms of admissible feature families, 
we obtain a unified setting that encompasses classical 
feedforward networks on Euclidean spaces as a special case.

We study the universal approximation property for topological neural networks with vector-valued outputs. In the arbitrary-width case, 
we establish general density results that extend classical
universality theorems to non-Euclidean spaces.

We also analyze deep narrow networks, where the
width of each hidden layer is uniformly bounded while depth is allowed to
grow. We identify conditions under which such networks
retain the universal approximation property, thereby extending deep narrow
universality beyond the Euclidean setting.

A consequence of this analysis is the explicit connection between
universality of deep narrow networks and classical results from topology and
dimension theory. In particular, Ostrand’s 
extension of the Kolmogorov superposition theorem
yields concrete feature maps for products of compact metric spaces.
As a consequence, we obtain explicit width bounds expressed in terms of
topological dimension. This demonstrates how geometric and topological
structure of the input space can be translated into 
approximation properties of the considered networks.

Several directions for further research remain open. One natural extension is
the study of neural networks with outputs in infinite-dimensional spaces.
Another direction is the investigation of quantitative approximation
rates and complexity bounds within the present topological framework. Finally,
it would be of interest to explore how the abstract feature-based approach
developed here interacts with practical applications.

\end{document}